\documentclass[12pt]{amsart} 
 
\usepackage{amssymb,amsmath,amsfonts,mathrsfs,mathtools} 
\usepackage{color} 
\usepackage{mathdots} 
\newtheorem{thm}{Theorem}[section] 
\newtheorem{prop}[thm]{Proposition} 
\newtheorem{lemma}[thm]{Lemma} 
\newtheorem{cor}[thm]{Corollary} 
\newtheorem{ass}[thm]{Assertion} 
 
\theoremstyle{definition} 
 
\newtheorem{example}[thm]{Example} 
 
\newtheorem{remark}[thm]{Remark} 
 
\newtheorem{prob}[thm]{Problem} 
 
\newtheorem{conjecture}[thm]{Conjecture} 
 
\numberwithin{equation}{section} 
 
\newcommand{\bull}{{\scriptscriptstyle \bullet}} 
\newcommand{\ibull}{i_\bull} 
\newcommand{\jbull}{j_\bull} 
\newcommand{\CC}{{\mathbb C}} 
\newcommand{\PP}{{\mathbb P}} 
\newcommand{\TT}{{\mathbb T}} 
\newcommand{\ZZ}{{\mathbb Z}} 
\newcommand{\Fl}{\mathbb{ F}\ell} 
\newcommand{\Flndot}{\Fl(\ndot)}

\newcommand{\prP}{{\textit pr\!}_P} 
 
\newcommand{\calV}{{\mathcal V}} 
 
\DeclareMathOperator{\id}{{\rm id}} 
 
\newcommand{\ndot}{n_\bull} 
\newcommand{\Fdot}{F_{\bull}} 
 
\newcommand{\Adot}{A_{\bull}} 
\newcommand{\topd}{{\rm top}} 
\newcommand{\snd}{{\rm snd}} 
 
\newcommand{\defcolor}[1]{{\color{blue}#1}} 
\newcommand{\demph}[1]{\defcolor{{\sl #1}}} 
 
\begin{document} 
 
\title[The fundamental group of open Richardson varieties]{On the fundamental group of\\ open Richardson varieties} 
 
\author{Changzheng Li} 
 \address{School of Mathematics, Sun Yat-sen University, Guangzhou 510275, P.R. China} 
\email{lichangzh@mail.sysu.edu.cn} 
\urladdr{http://math.sysu.edu.cn/gagp/czli} 
\author{Frank Sottile} 
\address{Department of Mathematics. 
Texas A\&M University 
College Station, TX 77843, USA} 
\email{sottile@math.tamu.edu} 
\urladdr{http://www.math.tamu.edu/\~{}sottile} 
 
\author{Chi Zhang} 
 \address{School of Mathematics, Sun Yat-sen University, Guangzhou 510275, P.R. China} 
\email{zhangch223@mail2.sysu.edu.cn; chizhmath@gmail.com} 
 
\thanks{} 
%
%

\begin{abstract} 
  We compute the fundamental group of an open Richardson variety in the manifold of complete flags that 
  corresponds to a partial flag manifold. 
  Rietsch showed that these log Calabi-Yau varieties underlie a Landau-Ginzburg mirror for the 
  Langlands dual partial flag manifold, and our computation verifies a prediction of Hori for this mirror. 
  It is log Calabi-Yau as it isomorphic to the complement of the Knutson--Lam--Speyer 
  anti-canonical divisor for the partial flag  manifold. 
  We also determine explicit defining equations for this divisor. 
\end{abstract} 
 
\maketitle 
 
%
\section{Introduction} 
It is an old problem of Zariski~\cite{Zar1} to compute the fundamental group of the complement of an algebraic curve in the 
complex projective plane. 
The fundamental group of the complement of a projective hypersurface 
reduces to the case of a plane curve by Zariski's Theorem of Lefschetz type~\cite{Zar2}. 
More generally, one may ask about the fundamental group of the complement of a divisor in a projective variety. 
Examples of importance in mirror symmetry are log Calabi-Yau varieties~\cite{GHK, HKP, KKP}, which are 
quasi-projective varieties that are the complement of 
an anti-canonical divisor in a smooth projective variety. 
We consider this case when the ambient projective variety is a flag variety. 
 
Let $G$ be a complex, simply-connected, simple Lie group with a Borel subgroup $B$. 
For an element $u$ in the Weyl group $W$ of $G$, the (opposite) Schubert cells 
$\mathring {X}_u$ and $\mathring X^u$ in $G/B$ 
are affine spaces of codimension and dimension $\ell(u)$ respectively, where 
$\ell\colon W\to \ZZ_{\geq 0}$ is the length function. 
The open Richardson variety 
\[ 
  \mathring X_v^w\ :=\  \mathring X_v\cap \mathring X^w 
\] 
is irreducible and has dimension $\ell(w){-}\ell(v)$ if $v\leq w$ in Bruhat order and otherwise it is empty. 
It is a log Calabi-Yau variety \cite{KLS}. 
We pose the following:
 
\begin{prob}\label{quespi1} 
   What is the fundamental group of $\mathring X_v^w$? 
\end{prob} 
 
Fundamental groups of log Calabi-Yau varieties arise in mirror symmetry, which is about equivalences of two apparently 
completely different physical theories. 
For instance, one mirror symmetry statement asserts that the small quantum cohomology of a Fano  manifold $Y$ should be 
isomorphic to the Jacobi ring of a holomorphic function  $f\colon Z\to \CC$ defined on an open Calabi-Yau variety 
$Z$ \cite{Bat, EHX, Giv}. 
Such pair $(Z, f)$ is a Landau-Ginzburg model mirror to $Y$. 
The Jacobi ring of $f$ is the coordinate ring of the critical points of $f$, and therefore the mirror space $Z$ is not 
uniquely determined. 
Nevertheless, physicists expect a mirror with certain optimal physical properties. 
According to Kentaro Hori\footnote{Personal communication and talks.}, one of these properties is manifested in the 
fundamental group, $\pi_1(Z)$, of $Z$ as follows. 
 
\begin{ass}\label{A:Hori} 
  Let $Y$ be a Fano manifold, and $D$ be a specified anti-canonical divisor on $Y$. 
  If {\upshape $\mbox{Aut}(Y, D)$} contains a maximal compact torus $(S^1)^m$, then an optimal mirror Landau-Ginzburg model 
  $(Z, f)$ should have  $\pi_1(Z)=\ZZ^m$. 
\end{ass} 
 
We consider this when $Z$ is an open Richardson variety $\mathring X^{w_0}_{w_P}$. 
Here, $P\supset B$ is a parabolic subgroup of $G$ and $w_0$ (resp.\ $w_P$) is the longest element in $W$ 
(resp.\ the Weyl group $W_P$ of the Levi subgroup of $P$). 
This is a log Calabi-Yau variety, as it is isomorphic to the complement of the Knutson-Lam-Speyer~\cite{KLS} anti-canonical 
divisor $-K_{G/P}$ in the flag manifold $G/P$. 
Let  $G^\vee$ (resp.\ $P^\vee$) denote the Langlands dual Lie group of $G$ (resp.\ $P$). 
Rietsch~\cite{Riet} constructed a Landau-Ginzburg model $(\mathring X_{w_P}^{w_0}, f)$ mirror to the flag 
manifold  $G^\vee/P^\vee$, assuming unpublished work of Peterson \cite{Pet}. 
This has been verified when  $G^\vee/P^\vee$ is a flag manifold of Lie type $A$~\cite{Riet2} and when it is either a 
minuscule  or a cominuscule flag variety~\cite{LaTe, PeRi, PRW}. 
The automorphism group  of $G^\vee/P^\vee$ is  $G^\vee$ (except for three special types of Grassmannians of  Lie type 
$B, C$, or $G_2$ which are homogeneous with respect to a larger simple Lie group)~\cite{Akhi}. 
The subgroup of $\mbox{Aut}(G^\vee/P^\vee)$ that preserves $-K_{G^\vee/P^\vee}$ is a complex torus $(\CC^\times)^{n-1}$, 
where $G$ has rank $n{-}1$. 
Following Assertion~\ref{A:Hori} and the belief that Rietsch's mirror is optimal, we expect that 
$\pi_1(\mathring X_{w_P}^{w_0})=\ZZ^{n-1}$. 
Our main result verifies this prediction when $G^\vee/P^\vee$ has Lie type $A$. 
 
\begin{thm}\label{mainthm} 
  Let $P$ be a parabolic subgroup of $SL(n, \CC)$. 
  Then $\pi_1(\mathring X_{w_P}^{w_0})=\ZZ^{n-1}$. 
\end{thm} 
 
A flag variety of Lie type $A$ is determined by a sequence $\ndot : 0<n_1<\dotsb<n_r<n$ of integers. 
The corresponding flag variety  $\Flndot$ is the set of all sequences of subspaces 
\[ 
   F_{n_1}\ \subset\ F_{n_2}\ \subset\ \dotsb\ \subset F_{n_r}\ \subset\ \CC^n 
   \quad\mbox{where}\quad  \dim F_i=i\,. 
\] 
This is a subvariety of the product of Grassmannians $G(n_1,n)\times\dotsb\times G(n_r,n)$. 
Under the Pl\"ucker embedding of $G(n_i,n)$ into the projective space $\PP^{\binom{n}{n_i}-1}$, the flag variety 
$\Flndot$ has a Pl\"ucker embedding into the product 
$\PP^{\binom{n}{n_1}-1}\times\dotsb\times\PP^{\binom{n}{n_r}-1}$. 
Although $\Flndot$ is a compactification of $\mathring X_{w_P}^{w_0}$ in this Pl\"ucker embedding, we prove 
Theorem~\ref{mainthm} by considering a different compactification of $\mathring X_{w_P}^{w_0}$ in a single projective 
space. 
This allows us to reduce Theorem~\ref{mainthm} to Zariski's classical case of a plane curve complement. 
We do this by investigating the intersections of the different irreducible components of the Knutson-Lam-Speyer~\cite{KLS} 
anti-canonical divisor $-K_{\Flndot}$, whose defining equations we also determine. 
 
A projected Richardson variety $\prP(X_v^w)$ is the image of a Richardson variety 
$X_v^w=  X_v\cap  X^w$ under the natural projection $\prP\colon G/B\to G/P$. 
This enjoys many geometric properties of Richardson varieties, such as being normal, Cohen-Macaulay, and 
having  rational singularities~\cite{BiCo,Brion,KLS}. 
The union of certain projected Richardson hypersurfaces forms an anti-canonical divisor $-K_{G/P}$ of 
$G/P$~\cite{KLS}. 
Another main result is explicit defining equations in Theorem~\ref{thmequantion} for these projected Richardson 
hypersurfaces in terms of the Pl\"ucker coordinates when $G=SL(n,\CC)$. 
Each is given either by a single Pl\"ucker variable or by a bilinear quadric. 
For instance, $\Fl(1,3; 4)\subset\PP^3\times\PP^3$ is the hypersurface 
$\calV(x_1x_{234}-x_2x_{134}+x_3x_{124}-x_4x_{123})$ and $-K_{\Fl({1,3; 4})}$ is the divisor 
$\calV(x_1x_{234}(x_1x_{234}-x_2x_{134})x_{4}x_{123})$. 
We expect these explicit defining equations  to also be helpful in the study of the mirror symmetry for 
$\Flndot$, similar to the study of mirror symmetry for Grassmannians in~\cite{MaRi}. 
 
The paper is organized as follows. 
We review basic facts on Richardson varieties in Section~\ref{S:OpenRicharadson}. 
We provide an expectation for the fundamental group $\pi_1(\mathring X_{\id}^{v})$ in Section~\ref{S:PiOne}. 
For $G=SL(n, \CC)$, we  derive  the explicit defining equations of $-K_{G/P}$ in terms of the Pl\"ucker coordinates in 
Section~\ref{S:KLS_Equations}, and then  compute the fundamental group of the complement $-K_{G/P}$ in $G/P$ in 
Section~\ref{S:MainResult}. 
Finally in Section~\ref{secproof}, we provide the proof of Lemma \ref{lemma:Irr}.

\section{Open Richardson varieties}\label{S:OpenRicharadson} 
 
Let \defcolor{$G$} be a complex, simply-connected, simple Lie group of rank $n{-}1$, and $\defcolor{B}\subset G$ be a Borel 
subgroup containing a maximal complex torus $\defcolor{\TT}\simeq(\CC^\times)^{n-1}$. 
Let $\defcolor{\Delta}=\{\alpha_1, \dotsc, \alpha_{n-1}\}$ be a basis of simple roots in $(\mbox{Lie}(\TT))^*$. 
The Weyl group \defcolor{$W$} of $G$ is a Coxeter group generated by  the  simple reflections 
$\{s_{\alpha}\mid \alpha\in \Delta\}$, and is identified with the quotient $N_G(\TT)/\TT$, where \defcolor{$N_G(\TT)$} is 
the normalizer of $\TT$ in $G$. 
For each $u\in W$, choose a lift $\defcolor{\dot{u}}\in N_G(\TT)$. 
The opposite Borel is $B^-:= \dot w_0B \dot w_0$, where $w_0$ is the longest element in $W$. 
The (opposite) Schubert cells 
\[ 
   \mathring {X}_u\ :=\ B^- \dot uB/B\ \cong\ \CC^{\dim G/B-\ell(u)} 
  \qquad\mbox{and}\qquad 
  \mathring X^u\ :\ =B \dot uB/B\ \cong\ \CC^{\ell(u)} 
\] 
are  independent of choice of lift $\dot{u}$. 
Henceforth, we write $u$ for $\dot u$. 
 
The root system of $(G, B)$ is $\defcolor{R}:=W\cdot\Delta=R^+\sqcup (-R^+)$, where 
$\defcolor{R^+}:=R\cap \bigoplus_{i=1}^{n-1}{\ZZ_{\geq 0}}\alpha_i$ is the set of positive roots. 
Each root  $\gamma=w(\alpha_i)\in R$ gives a reflection $s_\gamma:=ws_iw^{-1}\in W$, independent of the expressions for 
$\gamma$. 
The \demph{Bruhat order} on $W$ is the transitive closure of its covering relation, \defcolor{$u\lessdot v$} for $u,v\in W$ 
if $\ell(v)=\ell(u)+1$ and $v=us_\gamma$ for some  $\gamma\in R$, where $\ell\colon W\to \ZZ_{\geq 0}$ is the length 
function. 
The \demph{open Richardson variety} 
\[ 
  \defcolor{\mathring X_v^u}\ :\ =\mathring X_v\cap \mathring X^u 
\] 
is irreducible and of dimension $\ell(u)-\ell(v)$ if $v\leq u$, and otherwise it is empty. 
Its closure, a (closed) Richardson variety, is the intersection $\defcolor{X^u_v}:=X_v\cap X^u$ of (opposite) Schubert 
varieties $\defcolor{X_v}$ and $\defcolor{X^u}$, which are closures of the corresponding Schubert cells. 
As $w_0^2=\id$, we have the following identification of open Richardson varieties:
 
\begin{prop}\label{propXX} 
  For any $v\in W$, $\mathring X_v^{w_0}\cong \mathring X_{\id}^{w_0v}$. 
\end{prop} 
\begin{proof} 
   $\mathring X_v^{w_0}\              
   =\ w_0Bw_0vB/B\cap Bw_0B/B\ 
  \cong\  Bw_0vB/B\cap w_0Bw_0B/B\      
   =\ X^{w_0v}_{\id}$. 
\end{proof} 
 
A proper parabolic subgroup $P\supsetneq B$ determines and is determined by a proper subset 
$\defcolor{\Delta_P}\subsetneq\Delta$. 
The Weyl group \defcolor{$W_P$} of (the Levi subgroup) of $P$ is the subgroup of $W$ generated by 
$\{s_\alpha\mid \alpha \in \Delta_P\}$. 
Let \defcolor{$W^P$} be the set of minimal length coset representatives of $W/W_P$. 
We write \defcolor{$\prP$} for both the natural projection $G/B\rightarrow G/P$ and the map $W\rightarrow W^P$ determined 
by $w\in \prP(w)W_P$. 
Then $\prP(w_0)=w_0w_P\in W^P$, where  \defcolor{$w_P$} is the longest element in  $W_P$. 
Following~\cite{KLS}, the \demph{$P$-Bruhat order}, \defcolor{$\leq_P$}, is the suborder of the Bruhat order whose covers 
are \defcolor{$u\lessdot_P v$} when $u\lessdot v$ and $\prP(u)< \prP(v)$. 
The varieties 
\[ 
  \defcolor{\mathring \Pi_v^w}\ :=\ \prP(\mathring X_v^w)\quad 
  \mbox{and}\quad 
  \defcolor{\Pi_v^w}\ :=\ \prP(X_v^w) 
\] 
are  \demph{open} and \demph{closed projected Richardson varieties}, respectively. 
The next proposition is implicit in~\cite{KLS}. 
We explain how it follows from explicit results there. 
 
\begin{prop}\label{propanti} 
  The open Richardson variety $\mathring X_{w_P}^{w_0}$ is isomorphic to the complement in $G/P$  of 
\[ 
   \defcolor{-K_{G/P}}\ :=\ 
   \sum_{\id\lessdot u\leq_P w_0w_P} \prP(X^{w_0w_P}_{u})\ +\ 
   \sum_{\id\leq_P v\lessdot w_0w_P} \prP(X^{v}_{\id})\,, 
\] 
  which is an anti-canonical divisor of $G/P$. 
\end{prop} 
\begin{proof} 
  By Proposition \ref{propXX}, we have $\mathring X_{w_P}^{w_0}\cong\mathring X_{\id}^{w_0w_P}$. 
  As $\id\leq_P w_0w_P$, we have $\dim \Pi_{\id}^{w_0w_P}=\ell(w_0w_P)=\dim G/P$ by \cite[Corollary 3.2]{KLS}, 
  and hence $\Pi_{\id}^{w_0w_P}=G/P$. 
  By \cite[Lemma 3.1]{KLS},  $\mathring X_{\id}^{w_0w_P}\cong \mathring\Pi_{\id}^{w_0w_P}$. 
  By \cite[Proposition 3.6]{KLS},  we have  $\Pi_{\id}^{w_0w_P}\smallsetminus \mathring \Pi_{\id}^{w_0w_P}=-K_{G/P}$. 
  It follows again from   \cite[Proposition 3.6, Corollary 3.2]{KLS} that $-K_{G/P}$ is the sum of all projected Richardson 
  hypersurfaces in $\Pi_{\id}^{w_0w_P}$, and hence it is  an anti-canonical divisor of  $\Pi_{\id}^{w_0w_P}$ 
  by \cite[Lemma 5.4]{KLS}. 
\end{proof}

\section{Expectation for $\pi_1(\mathring X_{\id}^v)$}\label{S:PiOne} 
 
The open Richardson variety $\mathring X_{\id}^{w_0w_P}$ has the form $\mathring X_{\id}^v$ where $v\in W$. 
We begin with some well-known facts about fundamental groups. 
 
\begin{prop}[Zariski Theorem of Lefschetz type \cite{Zar2}]\label{ZariskiTHM} 
  Let $V$ be a hypersurface in $\PP^N$. 
  For almost every two-plane $\Lambda\subset \PP^N$, the map 
\[ 
    \pi_1(\Lambda\smallsetminus  V)\ \longrightarrow\  \pi_1(\PP^N\smallsetminus V) 
\] 
  induced by the inclusion $(\Lambda\smallsetminus V)\hookrightarrow(\PP^N\smallsetminus V)$ is an isomorphism. 
\end{prop} 
\begin{prop}[\cite{OkSa}]\label{proppi2} 
 Let $C_1$ and $C_2$ be algebraic curves in $\CC^2$. 
 Assume that the intersection $C_1\cap C_2$ consists of $d_1d_2$ points where $d_i$ is the degree of $C_i$. 
 Then the fundamental group $\pi_1(\CC^2\smallsetminus C_1\cup C_2)$ is isomorphic to the product  
 $\pi_1(\CC^2\smallsetminus C_1)\times\pi_1(\CC^2\smallsetminus C_2)$. 
\end{prop} 
 
Subvarieties $X$ and $Y$ of projective or affine space meet \demph{transversally} at a point $p\in X\cap Y$ if $p$ is 
a smooth point of each and the defining equations for the tangent spaces $T_pX$ and $T_pY$ are in direct sum. 
They meet transversally if they are transverse at every point of their intersection, which implies 
that $X\cap Y$ is smooth and of the expected dimension. 
They meet \demph{generically transversally} if the subset of points of $X\cap Y$ where they meet 
transversally is dense in every irreducible component of $X\cap Y$. 
The conditions in Proposition~\ref{proppi2} on the curves $C_1$ and $C_2$ is that they meet transversally. 
Indeed, by B\'ezout's Theorem, their projective completions meet in $d_1d_2$ isolated points, counted with 
multiplicity. 
As their intersection consists of $d_1d_2$ points, they are transverse at every point of their intersection. 
 
\begin{prop}[see e.g.\ Remark 2.13 (1) of \cite{LeMa}]\label{proppi3} 
 If $C$ is a smooth algebraic curve in $\CC^2$ whose projective completion is transverse to the line at 
 infinity, then  $\pi_1(\CC^2\smallsetminus C)=\ZZ$. 
\end{prop} 
 
Since $X_{\id}=G/B$, the Schubert cell $\mathring X_{\id}$ is the complement of the union of Schubert hypersurfaces 
$X_{s_\alpha}$ for $\alpha\in\Delta$. 
For $v\in W$, the Schubert cell $\mathring X^v\cong \CC^{\ell(v)}$. 
Therefore, 
 \begin{align*} 
     \mathring X_{\id}^v \ =\ \mathring X^v  \cap \mathring X_{\id}\ 
                          &=\  \mathring X^v \cap X_{\id}\;\smallsetminus\;  X^v\cap \partial X_{\id}\\ 
                          &=\ \CC^{\ell(v)}\;\smallsetminus\; X^v \cap \cup_{\alpha\in \Delta} X_{s_\alpha}\ 
                          =\ \CC^{\ell(v)}\;\smallsetminus\ \bigcup_{\alpha\in \Delta}X^v_{s_\alpha}\,. 
 \end{align*} 
The Richardson variety  $X_{s_\alpha}^v$ has dimension $\ell(v){-}1$ (and contains $\mathring X_{s_\alpha}^v$ as a Zariski 
open dense subset) if $s_\alpha\leq v$, and otherwise it is empty. 
A Richardson variety is reduced and normal, and thus its singular set has codimension at least two. 
Therefore, if $\Lambda\subset\CC^{\ell(v)}= \mathring X^v$ is a general affine two-plane, then 
$\defcolor{C_\alpha}:= X^v_{s_\alpha} \cap \Lambda$  is a smooth curve in $\Lambda$, whenever $s_\alpha\leq v$. 
If these curves satisfy the hypotheses of Propositions \ref{proppi2} and \ref{proppi3}, we are led to the following 
expectation. 
For any $v\in W$, define $\defcolor{\Gamma(v)}:=\{\alpha\in \Delta\mid  s_\alpha\leq v\}$. 
 
\begin{conjecture}\label{claimPi1} 
 We have $\pi_1(\mathring X_{\id}^v)=\ZZ^{|\Gamma(v)|}$. 
\end{conjecture} 

\begin{example} 
  The   flag manifold  $SL(3,\CC)/B=\{F_1\subset F_2\subset \CC^3\mid \dim F_i=i\}$ is the hypersurface 
  $\calV(x_{1}x_{23}-x_2x_{13}+x_3x_{12})$ in $\PP^2\times \PP^2$, where $[x_1,x_2,x_3]$ are coordinates for the first 
  $\PP^2$ and $[x_{12},x_{13},x_{23}]$ are coordinates for the second. 
  The Schubert cell $\mathring X_{\id}$ (resp.\ $\mathring X^{w_0}$) is the subset of this hypersurface 
  where $x_1 x_{12}\neq 0$ (resp.\  $x_3x_{23}\neq 0$). 
  Dehomogenizing by setting $x_1= x_{12}=1$, writing the remaining coordinates as $(z_2,z_3,z_{13},z_{23})\in\CC^4$, and 
  using the equation $0=z_{23}-z_2z_{13}+z_3$ to solve for $z_{23}$, we obtain 
\[ 
   \mathring X^{w_0}_{\id}\ =\ 
     \{(z_2, z_3, z_{13})\in \CC^3\mid z_3\neq 0, z_{2}z_{13}-z_3\neq 0\}\,. 
\] 
  This is the complement in $\CC^3$ of two smooth hypersurfaces whose intersection is transverse away from 
  $(0,0,0)$. 
  Intersecting with a general two-plane $\Lambda$ gives two smooth curves in $\Lambda$ that satisfy the hypotheses of 
  Propositions~\ref{proppi2} and~\ref{proppi3}. 
  Thus $\pi_1(\mathring X^{w_0}_{\id})=\ZZ^2$. 
 
  The Schubert subvariety $X^{s_1s_2}$ of $SL(3,\CC)/B$ is  $\calV(x_3,x_1x_{23}-x_2x_{13})$. 
  The Schubert cell $\mathring{X}^{s_1s_2}$ is the subset where $x_2x_{23}\neq 0$. 
  Setting $x_2=x_{23}=1$ and using $z_*$ for the remaining coordinates, 
  gives $\mathring X^{s_1s_2}=\{(z_1,z_{12},z_{13})\in\CC^3\mid z_1-z_{13}=0\}$. 
  Solving for $z_{13}$, we obtain $\mathring X^{s_1s_2}_{\id}=\{(z_1, z_{12})\in \CC^2\mid z_1z_{12}\neq 0\}$, which shows 
  that $\pi_1(\mathring X^{s_1s_2}_{\id})=\ZZ^2$. 
  Fundamental groups of the remaining open Richardson varieties in $SL(3,\CC)/B$ are as follows. 
  \[ 
  \begin{tabular}{|c||c|c|c|c|c|c|} 
     \hline 
     $v$ & $\mbox{id}$ & $s_1$ & $s_2$ & $s_1s_2$ & $s_2s_1$ & $w_0$ \\ \hline\hline 
       \raisebox{-5pt}{\rule{0pt}{16pt}} 
     $\Gamma(v)$ & $\emptyset$ & $\{\alpha_1\}$ & $\{\alpha_2\}$ 
     & $\{\alpha_1, \alpha_2\}$ & $\{\alpha_1, \alpha_2\}$ & $\{\alpha_1, \alpha_2\}$ \\ \hline 
       \raisebox{-5pt}{\rule{0pt}{18pt}} 
     $\pi_1(\mathring X_{\id}^v)$ & $\{\mbox{id}\}$ & $\ZZ$ & $\ZZ$ & $\ZZ^{2}$ & $\ZZ^{2}$ & $\ZZ^{2}$ \\ 
     \hline 
  \end{tabular} 
   \eqno{\diamond} 
  \] 
\end{example} 
 
We establish some lemmas that will help to rewrite the expression for $-K_{G/P}$ from Proposition~\ref{propanti}. 
They use basic facts about reflection groups as could be found in, for example~\cite[$\S 1$]{Humphreys}. 
 
\begin{lemma}\label{lemcox} 
  \begin{enumerate} 
    \item If $w=s_{i_1}\cdots s_{i_m}\in W^P$ is a reduced expression of $w$, then $s_{i_j}\cdots s_{i_m}$ is also in $W^P$ and is again a reduced expression (of length $(m-j+1)$). 
    \item If $\beta\in \Delta_P$ and $v\in W^P$ satisfy both $s_\beta\not\leq v$ and
      $s_\beta v\neq vs_\beta$, then $\ell(s_\beta v)=\ell(v)+1$ and $s_\beta v\in W^P$. 
  \end{enumerate} 
\end{lemma} 
\begin{proof} 
 (1) $w\in W^P$ if and only if $\ell(ws_\alpha)=\ell(w)+1$ for all $\alpha \in \Delta_P$. Since the given expression of $w$ is reduced, we have $\ell(s_{i_j}\cdots s_{i_m} 
  s_\alpha)=(m-j+1)+1=\ell(s_{i_j}\cdots s_{i_m})+1$ for any $\alpha\in \Delta_P$. Hence, $s_{i_j}\cdots s_{i_m}\in W^P$ and it is a reduced expression. 
 
  (2) Since $s_\beta \not\leq v$, any reduced expression of $v^{-1}$ does not contain $s_\beta$, and hence $v^{-1}(\alpha)\in R^+$. Thus $\ell(s_\beta v)=\ell(v^{-1}s_\beta)=\ell(v^{-1})+1=\ell(v)+1$. 
 
For any $\alpha\in \Delta_P$, we have $v(\alpha)\in R^+$ as $v\in W^P$;  we claim $s_\beta v(\alpha)\in R^+$ for all such $\alpha$ and hence $s_\beta v\in W^P$. Indeed, if $\alpha\neq \beta$, then we have $v(\alpha)\neq \beta$, as   any reduced expression of $v$ does not contain $s_\beta$. Moreover,  $v(\beta)\neq \beta$ (otherwise $vs_\beta v^{-1}=s_\beta$, contradicting to the hypothesis). Therefore the claim holds by noting  $s_\beta(R^+\setminus\{\beta\})=R^+\setminus\{\beta\}$. 
\end{proof} 
 
\begin{lemma}\label{lemmaGamma} 
    For any parabolic subgroup $P$, we have $\Gamma(w_0w_P)=\Delta$. 
\end{lemma} 
\begin{proof} 
 For any $\alpha\in \Delta\smallsetminus \Delta_P$, we have $w_P(\alpha)>0$ and thus $w_0w_P(\alpha)<0$. 
 Consequently, $w_0w_P$ has a reduced expression ending with $s_\alpha$ (by \cite[$\S 1.7$ Exchange Condition]{Humphreys}). 
 Thus $w_0w_P\geq s_\alpha$ and $\alpha\in \Gamma(w_0w_P)$. 
 It remains to show $\Delta_P\subset  \Gamma(w_0w_P)$. 
 
 If $\Delta_P\not\subset  \Gamma(w_0w_P)$, then  there exists $\alpha\in \Delta_P$ such that $s_\alpha\not\leq w_0w_P$. 
Since the Dynkin diagram of $\Delta$ is a tree, 
 there exist  $\{\beta_1,\ldots, \beta_m\}$ satisfying both  (1) $\beta_1=\alpha$,   $\{\beta_1,\ldots, \beta_{m-1}\}\subset \Delta_P$,   $\beta_m\in \Delta\setminus \Delta_P$,  and (2) $\beta_i$ is adjacent to $\beta_{i+1}$ for $i=1,\ldots, m-1$. 
 Then for $\gamma:=s_{\beta_1}\cdots s_{\beta_{m-1}}(\beta_m)=\sum_{j=1}^m a_j\beta_j$ with $a_j>0$ for all $j$, we have 
 $w_P(\gamma)=w_P(\sum_{j=1}^{m-1} a_j\beta_j)+w_P(a_m\beta_m)>0$ and consequently $w_0w_P(\gamma)<0$. 
 However, $w_0w_P$ is in the Weyl subgroup generated by $\{s_\beta\mid \beta\in \Delta\smallsetminus\{\alpha\}\}$, by the 
 hypothesis  $s_\alpha\not\leq w_0w_P$. 
 Thus we deduce a contradiction by noting   $w_0w_P(\gamma)=w_0w_P(a_1\alpha)+w_0w_P(\sum_{j=2}^ma_j\beta_j)>0$. 
\end{proof} 
 
For general $G/P$, the expectation $\pi_1(\mathring X_{\id}^{w_0w_P})\simeq \ZZ^{|\Delta|}$ would follow from 
Conjecture~\ref{claimPi1} and Lemma~\ref{lemmaGamma}. 
We refine the description of $-K_{G/P}$ of Proposition~\ref{propanti}. 
Moreover, we have the following. 
 
\begin{lemma}\label{lemmared} 
   Let $u\in W$. Then we have 
   \begin{enumerate} 
     \item ${\mbox{id}\lessdot u\leq_P  w_0w_P}$ if and only if $u=s_\alpha$ for some $\alpha\in \Delta$. 
     \item  ${\mbox{id}\leq_P u\lessdot w_0w_P}$ if and only if 
            $u=\prP(w_0s_\alpha)=w_0s_\alpha w_P$ for   $\alpha\in \Delta\smallsetminus \Delta_P$. 
   \end{enumerate} 
\end{lemma} 
\begin{proof} 
 If  ${{\id}\lessdot u}$, then $u=s_\alpha$ for some $\alpha\in \Delta$. 
 If $\alpha \in \Delta\smallsetminus \Delta_P$, then as in the beginning of the proof of Lemma~\ref{lemmaGamma}, 
 $w_0w_P$ admits a reduced expression $w_0w_P=s_{i_1}\cdots s_{i_l}$ where $l=\ell(w_0w_P)$ and $s_{i_l}=s_\alpha$. 
 For $j=1,\dotsc,l$, set $\defcolor{v_j}:=s_{i_{l-j+1}}\cdots s_{i_{l-1}}s_{i_l}$. 
 Then we have  $s_\alpha=v_1\lessdot\cdots\lessdot v_l=w_0w_P$ with $v_j\in W^P$ for any $j$ by Lemma \ref{lemcox} (1), which implies 
 $\prP(v_j)=v_j$. 
 Hence, $s_\alpha \leq_P w_0w_P$. 
 If $\beta\in \Delta_P$, then we still have $s_\beta\leq w_0w_P$ by Lemma \ref{lemmaGamma}, so $s_\beta$ must occur in the 
 aforementioned reduced expression of $w_0w_P$. 
 Let $m:=\max\{j\mid s_{i_j}=s_\beta\}$. 
 Set $u_j= s_{i_{l-j+1}}\cdots s_{i_{l-1}}s_{i_l}$ if $l-m+1\leq j\leq l$, 
 $\defcolor{u_j}= s_{i_m}s_{i_{l-j+2}}\cdots s_{i_{l-1}}s_{i_l}$ if $2\leq j\leq l-m$ (in which case $s_\beta u_j=s_{i_{l-j+2}}\cdots s_{i_l}\in W^P$ by Lemma \ref{lemcox} (1), and we will discuss whether it commutes with $s_\beta$), and $u_1=s_{i_m}$. 
 Then we have  $s_{\beta}=u_1\lessdot\cdots\lessdot u_l=w_0w_P$. 
 For $j\leq l-m$, we notice that  if $u_j=s_{i_m}u_js_{i_m}$, then $u_r=s_{i_m}u_rs_{i_m}$ and $s_{i_m}u_r\in W^P$ hold for 
 any $r\leq j$, and  if $u_j\neq s_{i_m}u_js_{i_m}$ then $u_j\in W^{P}$ by Lemma \ref{lemcox} (2). 
 It follows that $s_\beta\leq_P w_0w_P$ by definition. 
 
 If ${\id}\leq_P u$, then by definition we have $\ell(\prP(u))-\ell(\prP({\id}))\geq \ell(u)-\ell({\id})$, implying that 
 $\ell(\prP(u))\geq \ell(u)$ and hence $u\in W^P$. 
 Together with $u\lessdot {w_0w_P}$, it follows that $\ell(u)=\ell(w_0w_P)-1=\ell(w_0)-\ell(w_P)-1$, so that 
 $\ell(uw_P)=\ell(w_0)-1$. Hence, $uw_P=w_0s_\alpha$ for some $\alpha\in \Delta$. 
 This further implies $u=w_0s_\alpha w_P$, and hence 
 $\ell(w_0)-\ell(w_P)-1=\ell(u)=\ell(w_0s_\alpha w_P)=\ell(w_0)-\ell(s_\alpha w_P)$. 
 Therefore $\ell(s_\alpha w_P)=\ell(w_P)+1$, implying $\alpha\notin \Delta_P$. 
 On the other hand, for $\alpha\in \Delta\smallsetminus \Delta_P$, for any $\gamma\in R_P^+$, we have 
 $w_P(\gamma)\in R_P^-$, implying $s_\alpha w_P(\gamma) \in R^-$ and hence $w_0s_\alpha w_P(\gamma)\in R^+$. 
 Therefore $\prP(w s_\alpha)=w s_\alpha w_P\in W^P$ and ${\id}\leq_P ws_\alpha w_P$ for any such $\alpha$. 
\end{proof} 
 
\begin{prop}\label{propanti22} 
  ${\displaystyle -K_{G/P} \ =\ 
    \sum_{\alpha\in \Delta} \prP(X^{w_0w_P}_{s_\alpha})\ +\ 
   \sum_{\alpha\in \Delta\smallsetminus \Delta_P}  \prP(X^{w_0s_\alpha w_P}_{\id})}$. 
\end{prop} 
\begin{proof} 
   This is a direct consequence of Proposition \ref{propanti} and  Lemma \ref{lemmared}. 
\end{proof} 
 
\section{Defining equations of   $-K_{SL(n, \CC)/P}$} 
\label{S:KLS_Equations} 
Henceforth, we assume that $G=SL(n, \CC)$. 
Then $SL(n, \CC)/P=\defcolor{\Flndot}$ is the manifold of partial flags 
$\Fdot \colon F_{n_1}\subset\dotsb\subset F_{n_r}\subset \CC^n$ of type $\ndot$ ($\dim F_{n_i}=n_i$). 
Here $\defcolor{\ndot}:=1\leq n_1<\dotsb<n_r<n$ is an increasing sequence of integers and $P$ is the parabolic 
subgroup corresponding to the roots not in $\ndot$, so that $\defcolor{\Delta_P}=\{\alpha_i\mid i\not\in\ndot\}$. 
Also, $W=S_n$ is the symmetric group generated by simple transpositions $\{s_i\mid 1\leq i\leq n{-}1\}$.

The natural embedding of $\Fl(\ndot)$ into the product 
\[ 
  G(n_1,n)\ \times\ G(n_2,n)\ \times\ \dotsb\ \times\ G(n_r,n) 
\] 
of Grassmannians and then into the product $\PP(\wedge^{n_1}\CC^n) \times \dotsb \times\PP(\wedge^{n_r}\CC^n)$ of Pl\"ucker 
spaces gives Pl\"ucker coordinates \defcolor{$x_J$} for $\Flndot$. 
We describe their indexing. 
For any positive integer $m$, set $\defcolor{[m]}:=\{1,\dotsc,m\}$ and write \defcolor{$\binom{[m]}{j}$} for the set of 
subsets $J$ of $[m]$ of cardinality $j$, which we always write as increasing sequences. 
There is a  Pl\"ucker coordinate \defcolor{$x_J$} for $\Flndot$ for every $J\in\binom{[n]}{n_j}$, for each 
$j=1,\dotsc,r$. 
 
Let us explain $x_J$ concretely in terms of local coordinates for $\Flndot$. 
A point $\Fdot\in\Flndot$ is represented by a $n_r\times n$ matrix \defcolor{$\Adot$} of full rank $n_r$, where 
$F_{n_j}$ is the row space of the first $n_j$ rows of $\Adot$. 
For $J\in\binom{[n]}{n_j}$, the Pl\"ucker coordinate $x_J$ of $\Fdot$ is the determinant of the $n_j\times n_j$ submatrix 
of $\Adot$ formed by the first $n_j$ rows and the columns from $J$. 
This is the $J$th minor of the matrix formed by the first $n_j$ rows of $\Adot$. 
 
For $a<b\leq n$, let us write \defcolor{$(a,b]$} for the set $\{a{+}1,\dotsc,b\}$ and \defcolor{$[a,b)$} for 
$\{a,\dotsc,b{-}1\}$. 
Note that $(0,i]=[i]$. 
If $J\subset[a]$ and $J'\subset(a,n]$, then $J,J'$ is the index $J\cup J'\subset[n]$. 
 
Elements of $W^P$ index Schubert varieties in $\Flndot$, while elements of $S_n$ index Schubert varieties in $\Fl(n)=G/B$. 
The Richardson variety $X_{\id}^{w_0w_P}$ projects birationally onto $\Flndot$, under the map 
$\prP\colon  \Fl(n)\to \Flndot$. 
We describe explicit equations for the irreducible components of $-K_{\Flndot}$, which were identified in 
Proposition~\ref{propanti22}. 
 
\begin{thm}\label{thmequantion} Let $i\in [n{-}1]$. 
  \begin{enumerate} 
  \item For $i\in \ndot$, $\prP(X^{w_0s_i w_P}_{\id})$ is the Schubert divisor 
       of $\Flndot$ defined by the Pl\"ucker coordinate hyperplane $x_{(n-i, n]}=0$. 
  \item When $i\in \ndot$, $\prP(X^{w_0w_P}_{ s_i})$ is the Schubert divisor 
       of $\Flndot$ defined by  the Pl\"ucker coordinate hyperplane $x_{[i]}=0$. 
 \item When $i<n_1$, $\prP(X^{w_0w_P}_{s_i})$ is given by $x_{[i],(n-n_1+i, n]}=0$. 
 
 \item When $i>n_r$,  $\prP(X^{w_0w_P}_{s_i})$ is given by $x_{(i-n_r,i]}=0$. 
 
 \item When $n_j<i<n_{j+1}$ with $j\in[r{-}1]$, set $\defcolor{k}:=i-n_j$ and 
       $\defcolor{l}:=\min\{i, n{-}n_{j+1}{+}k\}$. 
 
   The projected Richardson hypersurface $\prP(X^{w_0w_P}_{s_i})$ is given  by 
   \begin{equation}\label{Eq:ProjectedRichardson} 
     \sum_{J\in\binom{[l]}{k}} (-1)^{|J|} x_{[i]\smallsetminus J}\cdot x_{J,(n-n_{j+1}+k,n]}\ =\ 0\ , 
 \end{equation} 
  where \defcolor{$|J|$} is the sum of the elements in $J$. 
\end{enumerate} 
\end{thm} 
\begin{proof} 
 We start with the most involved case (5). 
 As a first check, note that in~\eqref{Eq:ProjectedRichardson} the first Pl\"ucker coordinate 
 $x_{[i]\smallsetminus J}$ has $n_j$ indices, while 
 $x_{J,(n-n_{j+1}+k,n]}$ has $n_{j+1}$ indices. 
 To prove Statement (5), set $\defcolor{a_1}:=n_1$ and $\defcolor{a_i}:=n_{i}-n_{i-1}$ for $i=2,\dotsc,r$. 
 We start with a structured matrix parameterizing the Schubert cell $\mathring{X}^{w_0w_P}$, which has the block form 
 \begin{equation}\label{Eq:OlocalCoords} 
  \left(\begin{matrix} 
      * &   *    &\dotsc &   *   &    *   & I_{a_1}\\ 
      * &   *    &\dotsc &   *   & I_{a_2} &   0  \vspace{-4pt} \\ 
      * &   *    &\iddots&\iddots&   0    &   0  \vspace{-6pt}  \\ 
  \vdots& \vdots &\iddots&   0   &   0    &   0   \\ 
      * & I_{a_r} &   0   & \dotsb&   0    &   0 
   \end{matrix}\right)_{n_r\times n}\ . 
 \end{equation} 
Here, $I_{a}$ is the $a\times a$ identity matrix. 
Observe that the first column block has $n{-}n_r$ columns. 
The hypersurface Schubert variety $X_{s_i}$ in $G/B$ is defined by the single Pl\"ucker coordinate $x_{[i]}$, which is not 
a Pl\"ucker coordinate on $G/P$ when $i\not\in\ndot$. 
Our equation for $\prP(X_{s_i}^{w_0w_P})$ is obtained by evaluating $x_{[i]}$ on the  coordinates~\eqref{Eq:OlocalCoords} 
for $\mathring{X}^{w_0w_P}$ and expressing it in terms of the Pl\"ucker coordinates for $G/P$. 
 
To that end, suppose that $i\not\in\ndot$, and for now that $n_j<i<n_{j+1}$ as above. 
Consider the first $i$ rows of~\eqref{Eq:OlocalCoords}, 
 \[ 
  \left(\begin{matrix} 
      *  &  *   &      *     &\dotsc &   *   &   *  & I_{a_1}\\ 
      *  &  *   &      *     &\dotsc &   *   &I_{a_2}&   0   \vspace{-4pt}   \\ 
   \vdots&\vdots&   \vdots   &\iddots&\iddots&   0  &   0   \\ 
      *  &  *   &      *     & I_{a_j}&   0   &   0  &   0   \\ 
      *  & I_{k}&0_{k,a_{j+1}-k}&   0   &\dotsb &   0  &   0 
   \end{matrix}\right)_{i\times n}\ . 
 \] 
Here, \defcolor{$0_{k,a_{j+1}-k}$}  is the zero matrix with $k$ rows and $a_{j+1}{-}k$ columns, and the first column 
block  has size $n-n_{j+1}$. 
The Pl\"ucker coordinate $x_{[i]}$ is the determinant of the first $i$ columns of this matrix. 
Use Laplace expansion on the last $k$ rows to get 
\[ 
  x_{[i]}\ =\ \sum_{J\in\binom{[i]}{k}} 
  x_{[i]\smallsetminus J} \cdot z_J\ , 
\] 
where \defcolor{$z_J$} is the $J$th minor of the last $k$ rows, 
$\begin{pmatrix}* & I_k& 0_{k,a_{j+1}-k}&   0  & \dotsb \end{pmatrix}$. 
Its last nonzero column is in position $n{-}n_{j+1}{+}k$, so we may assume that $J\subset[l]$, 
as otherwise $z_J=0$. 
 
If we consider the form of the matrix~\eqref{Eq:OlocalCoords} (specifically, its first $n_{j+1}$ rows), then we see that 
$z_J = \pm x_{J,(n-n_{j+1}+k,n]}$, as the columns in the final positions in $(n-n_{j+1}+k,n]$ all end with a 1 in 
rows $1,\dotsc,n_j, n_j{+}k{+}1,\dotsc,n_{j+1}$. 
Rather than compute the sign, we note that the sign does not depend upon $J$, but only on $\ndot$ and $i$. 
Hence, $\prP(X_{s_i}^{w_0w_P})$ satisfies the formula~\eqref{Eq:ProjectedRichardson}, and then this completes the proof, by 
noting that the hypersurface of $G/P$ defined by \eqref{Eq:ProjectedRichardson} is irreducible. 
 
The arguments for cases (2), (3), (4) are similar and much simpler. 
Case  (1) follows from case (2) by noting  $\prP(X^{w_0s_i w_P}_{\id})=w_0 \prP(X^{w_0w_P}_{s_i})$ for $i\in \ndot$. 
\end{proof} 
 
\begin{example} 
 For $\Fl(3,6; 7)$, $\prP(X^{w_0w_P}_{s_4})$ is given by $x_{234}x_{134567}-x_{134}x_{234567}=0,$ 
 and  $\prP(X^{w_0w_P}_{s_5})$ is given by $x_{145}x_{234567}-x_{245}x_{134567}+x_{345}x_{124567}=0$. 
 \hfill$\diamond$ 
\end{example}

\section{Fundamental group of the complement of $-K_{\Flndot}$ in  $\Flndot$} 
\label{S:MainResult} 
 
To  study $\Flndot\smallsetminus(-K_{\Flndot})$, we first remove the Schubert divisors (1) in 
Theorem~\ref{thmequantion}. 
These are given by the Pl\"ucker coordinates $x_{(n-n_j, n]}$ for $n_j\in\ndot$, and 
correspond to the second sum in Proposition~\ref{propanti22}. 
This leaves the dense Schubert cell of $\Flndot$, which is identified with $\mathring {X}^{w_0w_P}$, 
and is parameterized by the coordinates~\eqref{Eq:OlocalCoords}. 
Let $\defcolor{N}:=\ell(w_0w_P)$ so that $\mathring{X}^{w_0w_P}\simeq\CC^N$, and let 
$\defcolor{\PP^N}:=\CC^N\sqcup\PP(\CC^N)$ be its projective completion. 
 
For any subvariety $D$ of $\Flndot$ which meets the cell $\mathring{X}^{w_0w_P}$, we also write $D$ for its 
closure in $\PP^N$. 
Write \defcolor{$D_0$} for the hyperplane $\PP(\CC^N)$ at infinity and for $i=1,\dotsc,n{-}1$, let 
\defcolor{$D_i:=\prP(X^{w_0w_P}_{s_i})\cap \mathring X^{w_0w_P}$} be 
the image of a projected Richardson variety that meets $\mathring{X}^{w_0w_P}$. 
Set $\defcolor{D}:=D_0\cup D_1\cup\dotsb\cup D_{n-1}$, a divisor in $\PP^N$. 
 
\begin{thm} 
  The fundamental group of $\Flndot\smallsetminus(-K_{\Flndot})$ is $\ZZ^{n-1}$. 
\end{thm} 
 
\begin{proof} 
 Since $\PP^N{\smallsetminus}D \simeq \Flndot\smallsetminus(-K_{\Flndot})$, we study the fundamental group of 
 the hypersurface complement $\PP^N\smallsetminus D$. 
 Let $\Lambda\subset\PP^N$ be a general two-plane. 
 For $i=0,\dotsc,n{-}1$, set $\defcolor{C_i}:=\Lambda\cap D_i$ and 
 set $\defcolor{C}:=\Lambda\cap D$, which are curves, as $\Lambda$ is general. 
 We claim that: 
 \begin{enumerate} 
  \item  Each curve $C_i$ is smooth. 
  \item  For $i\neq j$, then intersection $C_i\cap C_j$ is transverse. 
  \item  For $i,j,k$ distinct $C_i\cap C_j\cap C_k=\emptyset$. 
 \end{enumerate} 
 Given these claims, Propositions~\ref{proppi2} and~\ref{proppi3} imply that 
 $\pi_1(\Lambda{\smallsetminus}C)=\ZZ^{n-1}$, and Proposition \ref{ZariskiTHM} implies 
 $\pi_1(\PP^N\smallsetminus D)=\ZZ^{n-1}$, which implies the theorem. 
 
 By Bertini's Theorem and the genericity of $\Lambda$, these three properties of the curves $C_i$ are consequences of the 
 following three properties of the divisors $D_i$. 
\begin{enumerate} 
  \item  Each $D_i$ is smooth in codimension 1. 
  \item  For $i\neq j$, the intersection $D_i\cap D_j$ is generically transverse. 
  \item  For $i,j,k$ distinct, the intersection $D_i\cap D_j\cap D_k$ has codimension three. 
\end{enumerate} 
 
The hyperplane $D_0$ at infinity in $\PP^N$ is smooth. 
For $0<i$ the intersection $D_i\cap \CC^N$ with the complement of $D_0$ is 
isomorphic to an open part of the projected  Richardson variety $\prP(X^{w_0w_P}_{s_i})$ in $\Flndot$. 
Projected Richardson varieties are  normal~\cite{BiCo, Brion, KLS}, and thus smooth in codimension 1. 
Therefore, the first property is satisfied. 
 
For the second, we notice that for $0<i<j$, the intersection $X_{s_i}\cap X_{s_j}$ is given by 
$X_{s_is_j}$ if $j>i+1$, or $X_{s_is_j}\cup X_{s_js_i}$ if $j=i+1$, and in either case the intersection is reduced. 
%
%
%
%
It follows from the defining equations that 
$\prP(X^{w_0w_P}_{s_i})\cap \prP(X^{w_0w_P}_{s_j})= \prP(X^{w_0w_P}_{s_i} \cap  X^{w_0w_P}_{s_j})$, and hence the 
intersection is given by $\prP(X^{w_0w_P}_{s_is_j})$ if $j>i+1$, or 
$\prP(X^{w_0w_P}_{s_is_j})\cup \prP(X^{w_0w_P}_{s_js_i})$ if $j=i+1$. 
Thus the intersection  $D_i\cap D_j$ is generically transverse for $0<i<j$. 
 
To show that $D_0\cap D_i$ is generically transverse, we study the equations for $D_i$. 
The divisor $D_i$ is defined by the determinant $f^{(i)}$ of the upper left $i\times i$ submatrix of the local 
coordinates~\eqref{Eq:OlocalCoords}. 
Write $f^{(i)}$ as a sum of homogeneous pieces, 
\[ 
  f^{(i)}\ =\ f^{(i)}_{d_i}\ +\ f^{(i)}_{d_i-1}\ +\ \dotsb\ +\ f^{(i)}_0\,, 
\] 
where $\deg f^{(i)}_j = j$ and $\deg f^{(i)}= d_i$. 
If $z$ is a new homogenizing variable, so that $z=0$ defines the hyperplane $D_0$  at infinity in $\PP^N$, then 
$D_i$ is defined in $\PP^N$ by 
 \begin{equation}\label{Eq:homogenize} 
  f_{d_i}^{(i)}\ +\ zf_{d_i-1}^{(i)}\ +\ z^2 f_{d_i-2}^{(i)}\ +\ \dotsb\ +\ z^{d_i}f_0^{(i)}\,. 
 \end{equation} 
 
\begin{lemma}\label{lemma:Irr} 
  With these definitions, we have the following. 
   \begin{enumerate} 
   \item  For each $i\in [n{-}1]$, the top homogeneous component  $f^{(i)}_{d_i}$ of $D_i$ is square-free. 
       If $f^{(i)}$ is inhomogeneous, then its second highest homogeneous component $f^{(i)}_{d_i-1}$  is 
       nonzero and coprime to $f^{(i)}_{d_i}$. 
  \item  For $i, j\in [n{-}1]$ with $i\neq j$, the top homogeneous components of $f^{(i)}$ and $f^{(j)}$ are coprime. 
\end{enumerate} 
\end{lemma} 
 
We will prove this in Section \ref{secproof} and assume it for now. 
Then $D_0\cap D_i=\calV(z,f^{(i)})$ is defined in $D_0$ by the top homogeneous component  $f^{(i)}_{d_i}$  of $f^{(i)}$. 
Since  $f^{(i)}_{d_i}$ is square-free, $D_0\cap D_i$ is reduced in the plane $D_0$. 
When  $f^{(i)}$ is homogeneous, this shows that the intersection is generically transverse. 
When $f^{(i)}$ is inhomogeneous, the intersection will be nontransverse where $\calV(f^{(i)}_{d_i})$ is singular, and at 
points of $\calV(f^{(i)}_{d_i},f^{(i)}_{d_i-1})$. 
Since $f^{(i)}_{d_i}$ and $f^{(i)}_{d_i-1}$ are coprime, we see again that the intersection is generically transverse. 
 
Consider now the final point, that for $i<j<k$, $D_i\cap D_j\cap D_k$ has codimension three. 
If  $i\neq 0$, then this  follows from the same fact about the Richardson divisors. 
We may also see this from the defining equations, which give the following four cases.  
\newcommand{\pprop}{\raisebox{-6pt}{\rule{0pt}{17pt}}} 
\[ 
  \begin{array}{|c|l|}\hline 
   i,j,k &{}\qquad D_i\cap D_j\cap D_k\\\hline 
   i<j-1<k-2&\prP(X^{w_0w_P}_{s_is_js_k})\pprop\\\hline 
   i=j-1<k-2&\prP(X^{w_0w_P}_{s_is_js_k})\cup \prP(X^{w_0w_P}_{s_js_is_k})\pprop\\\hline 
   i<j-1=k-2&\prP(X^{w_0w_P}_{s_is_js_k})\cup \prP(X^{w_0w_P}_{s_is_ks_j})\pprop\\\hline 
   i=j-1=k-2&\prP(X^{w_0w_P}_{s_is_js_k})\cup \prP(X^{w_0w_P}_{s_js_is_k})\cup 
              \prP(X^{w_0w_P}_{s_ks_is_j})\cup \prP(X^{w_0w_P}_{s_ks_js_i})\pprop\\\hline 
   \end{array} 
\] 
 If $i=0$, then this holds as $f^{(j)}_{d_j}$ and $f^{(k)}_{d_k}$ are coprime. 
\end{proof} 
 
\section{Proof of Lemma~\ref{lemma:Irr}}\label{secproof}

Let $M$ be a principal $a\times a$ submatrix of~\eqref{Eq:OlocalCoords}. 
We will later show that its determinant equals the determinant of a matrix with a block 
form~\eqref{eqnmatrix} described below. 
Consequently, we first investigate the factorization of the top homogeneous component of the determinant of such a matrix, 
and use that to deduce Lemma~\ref{lemma:Irr}. 
Until we deduce Lemma~\ref{lemma:Irr} at the end of this section, all symbols, $N$, $r$, etc.\ will have different meanings 
than in Sections~\ref{S:KLS_Equations} and~\ref{S:MainResult}. 
We start with a well-known fact, as we will use similar arguments later. 
 
\begin{lemma}\label{lemma1} 
   The determinant $\det\big(x_{ij}\big)_{a\times a}$ of a matrix of indeterminates is irreducible. 
\end{lemma} 
\begin{proof} 
 Let $g$ be this determinant, and note that it has degree one in every variable $x_{ij}$. 
 Suppose that $g=pq$. 
 We may assume that $x_{11}$ appears in $p$, so that $p$ is of degree one in $x_{11}$. 
 Then $x_{1j}$ appears in $p$ for all $j$, for otherwise $x_{1j}$ appears in $q$, which implies that $x_{11}x_{1j}$ appears 
 in $g$, which is a contradiction. 
 Similarly, $x_{j1}$ appears in $p$, and then similar arguments show that each $x_{jk}$ appears in $p$. 
 Consequently,  $q$ is constant.	 
\end{proof} 

For sequences $\defcolor{\ibull}:=(i_1,i_2,\ldots,i_r)$ and $\defcolor{\jbull}:=(j_1,j_2,\ldots,j_r)$ of positive 
integers with $|\ibull|=|\jbull|=N$, consider a matrix of the following form, 
 \begin{equation}\label{eqnmatrix} 
  \defcolor{M(\ibull; \jbull)}\ :=\ \left(\begin{matrix} 
     *   & \dotsb &\dotsb &   *   &    *  & A^1 \\ 
     *   & \dotsb &\dotsb &   *   &  A^2  & I_{i_2,j_1}\vspace{-3pt}  \\ 
  \vdots & \dotsb &\iddots&\iddots&\iddots&0 \\ 
  * &\dots & A^s &I_{i_s, j_{s-1}} &0 &0 \vspace{-3pt} \\ 
  \vdots &\iddots&\iddots&\iddots &0&0      \\ 
    A^r  &I_{i_r, j_{r-1}}& 0 &\dotsb &0 &0 
 \end{matrix}\right)_{N\times N}. 
\end{equation} 
Here, \defcolor{$I_{c, d}$} is a $c\times d$ matrix with 1s on its diagonal and 0s elsewhere, 
the blocks \defcolor{$A^r$} are $i_r\times j_r$ matrices of indeterminates, and every $*$ 
denotes another matrix of indeterminates. 
As the entries of $M(\ibull; \jbull)$ that are not specified to be 0 or 1 are different indeterminates, all properties of 
its determinant $\defcolor{g}=\defcolor{g(\ibull;\jbull)}$ depend only upon the sequences $\ibull$ and $\jbull$. 
This includes whether or not $g=0$, its degree, its irreducibility and factorization, as well as the same properties of its 
top degree homogeneous component. 
 
We need not determine whether $g=0$, or if it is irreducible, or its degree. 
We do study the factorization of its top degree homogeneous component. 
For this, we set 
\[ 
  \defcolor{\Upsilon(\ibull;\jbull)}\ :=\ 
  \{s\in [r{-}1]\mid i_1+\cdots+i_s=j_1+\cdots +j_s\}\,. 
\] 
We show that this set controls the factorization of the top homogeneous component 
of the determinant $g$ of $M$. 
For a polynomial $f$, let \defcolor{$\topd(f)$} be the top homogeneous component of $f$ and 
\defcolor{$\snd(f)$} be the homogeneous component of $f$ of degree $\deg(f){-}1$.

\begin{lemma} \label{lemma2} 
  Let $g$ be the determinant of the matrix $M(\ibull; \jbull)$~\eqref{eqnmatrix}. 
  Assume that $g$ is irreducible and nonzero. 
  Then  $\topd(g)$  is reducible if and only if $\Upsilon(\ibull; \jbull)\neq \emptyset$. 
 \end{lemma} 
 \begin{proof} 
 Suppose that $\Upsilon(\ibull; \jbull)\neq\emptyset$. 
 Let $s\in\Upsilon(\ibull; \jbull)$ and observe that removing $I_{i_{s+1}, j_s}$ from~\eqref{eqnmatrix}  gives a 
 block upper left triangular matrix, $\left(\begin{smallmatrix}*&*\\ {*}&0\end{smallmatrix}\right)$. 
 Using Laplace expansion of $g$ along the first $i_1+\dotsb+i_s$ rows of $M(\ibull; \jbull)$, gives 
\begin{align*} 
  g\ =\ &\pm 
    \left|\begin{matrix}{} 
    \dotsb&*&* & A^{s+1} \\ 
    * &*&A^{s+2} & I_{i_{s+2},j_{s+1}}  \\ 
    \vdots &\iddots &\iddots &0 \\ 
     A^r&I_{i_r,j_{r-1}} &0 &0 
   \end{matrix}\right|\cdot 
   \left|\begin{matrix}{} 
     \dotsb&*&* & A^1 \\ 
     * &*&A^2 & I_{i_2,j_1}  \\ 
     \vdots &\iddots &\iddots &0 \\ 
     A^s &I_{i_s,j_{s-1}} &0 &0 
     \end{matrix}\right| \\ 
  & \  +\  \mbox{(other terms). \rule{0pt}{14pt}} 
 \end{align*} 
 In the other terms, at least one column of the first minor is from the lower right submatrix 
 \[ 
   \left(\begin{matrix} {}I_{i_{s+1}, j_s}& 0\\0&0 \end{matrix}\right)\ , 
 \] 
 and thus its degree is strictly less than that of the first minor in the first term. 
 Indeed, the minor is zero if any column is zero, and if not, then expanding that minor along the rows  containing 1s from 
 $I_{i_{s+1}, j_s}$ shows that its degree drops by the number of such rows/columns. 
 However, the second minor has the same degree as the second minor in the first term (as they have 
 the same format $M(i_1,\dotsc,i_s;j_1,\dotsc,j_s)$). 
 Since we assumed that $g\neq 0$, these second minors are all nonzero, and we conclude that the degree of the other terms 
 is strictly less than that of the first term. 
 Therefore, $\topd(g)$ is given by the product of top homogeneous components of the two minors in the first term of $g$, 
 neither  of which is a  constant (we see this by Laplace expansion along their first rows of indeterminates). 
 Thus $\Upsilon(\ibull; \jbull)\neq \emptyset$ is sufficient for 
 the  reducibility of $\topd(g)$. 
 
 We use induction on $r$ for necessity. 
 If $r=1$, then we are done by Lemma~\ref{lemma1}. 
 Suppose that for any sequences $\ibull$ and $\jbull$ of length $s<r$ with 
 $i_1+\dotsb+i_s=j_1+\dotsb+j_s$, if  $g(\ibull;\jbull)$ is irreducible and $i_1+\dotsb+i_t\neq j_1+\dotsb+j_t$ for all $1\leq t < s$, then 
 $\topd(g(\ibull;\jbull))$ is irreducible. 
 
 Let $\ibull$ and $\jbull$ be sequences of length $r$ such that 
 $i_1+\dotsb+i_s\neq j_1+\dotsb+j_s$ for any $1\leq s<r$, but 
 $i_1+\dotsb+i_r = j_1+\dotsb+j_r=N$. 
 Note that this implies that $i_r\neq j_r$. 
 
 Assume $i_r< j_r$. 
 Consider the Laplace expansion of $g$ along the last $i_r$ rows of $M(\ibull;\jbull)$. 
 For each $L\in\binom{[N]}{i_r}$, write \defcolor{$C_L$} for the determinant of the square submatrix formed by the 
 columns from $L$ and the last $i_r$ rows, and let \defcolor{$\hat{C}_L$} be its cofactor (determinant of the square 
 submatrix formed by the columns from $[N]\smallsetminus L$ and the first $N-i_r$ rows, with the appropriate 
 sign). 
 If $\defcolor{b}:=\min\{i_r,j_{r-1}\}$, then 
 \begin{equation}\label{Eq:g-expansion} 
   g\ =\ \sum_{L\in\binom{[j_r+b]}{i_r}} C_L \hat{C}_L\ =\ 
   \sum_{L\in\binom{[j_r]}{i_r}} C_L \hat{C}_L \ +\  \mbox{ (other terms)}\,. 
 \end{equation} 
 (The first sum is restricted as these are the only nonzero columns in the last $i_r$ rows.) 
 In the second expression, the degree of each of the (other terms)  is strictly less than the degree of the 
 terms in the sum over $L\in\binom{[j_r]}{i_r}$. 
 Indeed, in each, the minor $C_L$ has degree $|L\cap[j_r]|<i_r$ as $L$ includes at least one column beyond the 
 $j_r$th. 
 Thus  these minors have smaller degree than those in the sum over $\binom{[j_r]}{i_r}$. 
 Also, each cofactor $\hat{C}_L$ in either expression is, up to a sign, the determinant of a matrix of the 
 form~\eqref{eqnmatrix} with indices 
 \begin{equation}\label{Eq:newIndices} 
   M(i_1, \ldots, i_{r-2}, i_{r-1}\,;\, j_1, \ldots, j_{r-2}, j_{r-1}{+}j_r{-}i_r)\,. 
 \end{equation} 
 Thus they are either all zero or all nonzero. 
 As $g\neq 0$, we have $\hat C_L\neq 0$ for all $L$ and they all have the same degree and are irreducible. 
 Indeed, suppose that for some $L$, $\hat{C}_L=pq$ factors with neither $p$ nor $q$ a constant. 
 Since $\hat{C}_L\neq 0$, every entry  in the lower left $i_{r-1}\times (j_{r-1}+j_{r}-i_{r})$ submatrix of the matrix for 
 $\hat{C}_L$ appears in $\hat{C}_L$, which we may see by Laplace expansion  along its last $i_{r-1}$ rows.  
 If one entry occurs in $p$, then the argument used in the proof of Lemma \ref{lemma1} implies that they all do, and no 
 such entry occurs in $q$. 
 But then $q$  depends only on the last $(j_1+\cdots+j_{r-2})$ columns of the matrix for $\hat{C}_L$. 
 Since all the $\hat{C}_L$ have the same form~\eqref{Eq:newIndices}, they are all reducible with the same factor $q$. 
 But this implies that $q$ divides $g$, contradicting the irreducibility of $g$.

 As each $C_L$ for $L\in \binom{[j_r]}{i_r}$ is homogeneous of degree $i_r$, we have 
 \[ 
   \topd(g)\ =\  \sum_{L\in\binom{[j_r]}{i_r}} C_L \cdot \topd(\hat{C}_L)\,. 
 \] 
 Each term of some minor $C_L$ occurs only in that minor, and therefore appears in $\topd(g)$. 
 In particular, every indeterminate entry $x_{st}$ of the matrix $A^r$ occurs in $\topd(g)$. 
 We note that for each $L$, $\topd(\hat C_L)$ is irreducible, by our induction hypothesis, as $\hat C_L$ is irreducible and the corresponding sequences 
 in~\eqref{Eq:newIndices} have length $r{-}1<r$ and unequal partial sums.

 Suppose that $\topd(g)=pq$ factors as a product of polynomials. 
 We may assume that $x_{11}$ appears in $p$. 
 Arguing as in the proof of Lemma  \ref{lemma1} shows that each entry of  $A^r$ appears in $p$, and none appears 
 in $q$.

 For $L\in \binom{[j_r]}{i_r}$, let \defcolor{$y_L$} be the specialization obtained by replacing 
 $A^r$ by a matrix whose only nonzero entries form the identity matrix in the columns of $L$. 
 Since $A_K(y_L)=\delta_{K,L}$, the Kronecker delta, if we evaluate $\topd(g)$ at this specialization, we obtain 
 \[ 
   \topd(\hat{C}_L)\ =\ \topd(g)(y_L)\ =\ p(y_L)\cdot q(y_L)\ =\ p(y_L)\cdot q\,. 
 \] 
 Since $\topd(\hat{C}_L)$ is irreducible, if $q$ is nonconstant, then $p(y_L)$ is a nonzero constant. 
 Thus for $K,L \in\binom{[j_r]}{i_r}$ with $K\neq L$, we have 
\[ 
   p(y_L)\cdot \topd(\hat{C}_{K})\ =\  p(y_K)\cdot \topd(\hat{C}_L)\,, 
\] 
 which is a contradiction, as $\topd(\hat{C}_{K})$ and  $\topd(\hat{C}_L)$ have different indeterminates. 
 (Expand $\hat{C}_L$ along a column of $K\smallsetminus L$, whose indeterminates do 
 not appear in $\hat{C}_{K}$.)

 Suppose that $i_r> j_r$. 
 We prove that $\topd(g)$ is irreducible by modifying the argument for the case $i_r<j_r$. 
 Since $g$ is nonzero and irreducible, the matrix $M(\ibull;\jbull)$ does not contain a 
 $l\times(N-l)$ submatrix of zeroes, 
 for any $l$ (containing such a submatrix implies that $M(\ibull;\jbull)$ is upper left triangular so that $g$ factors, and 
 a larger submatrix forces $g$ to be zero). 
 Considering the last $i_r$ rows of $M(\ibull;\jbull)$, this implies that $i_r<j_r{+}j_{r-1}$. 
 
 The lower left $i_r\times (j_r{+}j_{r-1})$-corner of  $M(\ibull;\jbull)$ is 
 $(A^r \  \ I_{i_r,j_{r-1}} )$. 
 This  has $j_r{+}b$ nonzero columns where $b=\min\{i_r,j_{r-1}\}$. 
 Let us reconsider the expansion~\eqref{Eq:g-expansion} of $g$, 
\[ 
   g\ =\ \sum_{L\in\binom{[j_r+b]}{i_r}} C_L \hat{C}_L\ =\ 
   \sum_{L\in\binom{[j_r+b]}{i_r}\ ,\  [j_r]\subset L} C_L \hat{C}_L \ +\  \mbox{ (other terms)}\,. 
\] 
 The cofactors $\hat{C}_L$ as before are nonzero, have the same degree, and are irreducible. 
 The degree of $C_L$ is $|L\cap [j_r]|$, so only the terms in the sum in the second expression contribute to $\topd(g)$. 
 The rest of the argument proceeds as before. 
 \end{proof} 

We deduce three corollaries from this proof. 
In all, $g=\det M(\ibull;\jbull)$ is assumed to be nonzero and irreducible. 
Suppose that $\Upsilon(\ibull;\jbull)=\{s_1<\dotsb<s_m\}\neq\emptyset$. 
Set $s_0:=0$ and $s_{m+1}:=r$. 
For $t=0,\dotsc,m$, let 
 \begin{equation}\label{Eq:M_t} 
  M(\ibull; \jbull)_t\ :=\ 
   \left(\begin{matrix}{} 
  *&\cdots &* & A^{1+s_t}\\                         
  * &\cdots &A^{2+s_t} & I_{i_{2+s_t},j_{1+s_t}}  \\   
  \vdots &\iddots &\iddots &0 \\ 
   A^{s_{t+1}}&I_{i_{s_{t+1}}, j_{s_{t+1}-1}} &0 &0    
   \end{matrix}\right)\ , 
 \end{equation} 
which is a square submatrix of $M(\ibull; \jbull)$. 

\begin{cor}\label{corred} 
  If $\Upsilon(\ibull; \jbull)\neq\emptyset$, 
  then the irreducible factorization of $\topd(g)$ is $f_0\dotsb f_m$, where 
  $f_t=\topd(\det( M(\ibull; \jbull)_t))$. 
\end{cor} 
\begin{proof} 
  That $\topd(g)=f_0\dotsb f_m$ is a consequence of the proof of sufficiency in Lemma \ref{lemma2}. 
  The irreducibility of each $f_t$ is a consequence of the proof of necessity (using mathematical induction and arguing as for the irreducibility of $\hat C_L$ therein). 
\end{proof} 
 
\begin{remark}\label{R:reformulation} 
 When $s\in\Upsilon(\ibull; \jbull)$, let $m:=i_1+\dotsb+i_s$. 
 Then the matrix $M(\ibull; \jbull)$ has a block structure 
 \begin{equation}\label{Eq:Block} 
    \left(\begin{matrix} *&M\\M'&P\end{matrix}\right)\ , 
 \end{equation} 
 where $*$ is a $m\times(N{-}m)$ matrix of indeterminates, $M$ and $M'$ are structured matrices~\eqref{eqnmatrix} with 
 parameters 
 \[ 
   M\ =\ M(i_1,\dotsc,i_s; j_1,\dotsc,j_s) 
  \qquad 
   M'\ =\ M(i_{s+1},\dotsc,i_r; j_{s+1},\dotsc,j_r)\,, 
 \] 
 and $P$ is a  $(N{-}m)\times m$ matrix with block structure 
 $\left(\begin{smallmatrix}I&0\\ 0&0\end{smallmatrix}\right)$, where $I=I_{i_{s+1},j_{s}}$. 
 In particular, the $2\times 2$ submatrix on the anti-diagonal in rows $m,m{+}1$ 
 (and columns $n{-}m, n{-}m{-}1$) is  $\left(\begin{smallmatrix}*&*\\ *&1\end{smallmatrix}\right)$, where $*$ indicates an 
 indeterminate. 
 In particular, 
\[ 
   \topd(\det M(\ibull; \jbull))\ =\ \topd(\det M)\cdot\topd(\det M')\,. \eqno{\diamond} 
\] 
\end{remark} 

\begin{cor}\label{corVars} 
 Every indeterminate in each matrix $A^k$ for $k=1,\dotsc,r$ appears in $\topd(g)$. 
\end{cor} 
\begin{proof} 
 This can be proven by induction on $k$, using the same arguments as in the proof of necessity in Lemma \ref{lemma2}. 
\end{proof} 
 
\begin{cor} \label{corTop} 
 If $r=1$, then $g=\topd(g)$ is homogeneous and if $r>1$, then  $\snd(g)\neq 0$. 
\end{cor} 
\begin{proof} 
 If $r=1$, then $i_1=j_1=N$, and $g=\det A^1$ is a homogeneous polynomial. 
 
 Assume that $r>1$. 
 Expand $g$ along the last $i_r$ rows of  $M(\ibull; \jbull)$ as in the proof of necessity in Lemma \ref{lemma2}, 
 \[ 
   g\ =\ \sum_{L\in\binom{[j_r+b]}{i_r}} C_L\, \hat{C}_L\,. 
 \] 
 Recall that $C_L$ is homogeneous of degree $|L\cap[j_r]|$ and that $\hat{C}_L$ is the determinant of a matrix with 
 format~\eqref{Eq:newIndices}, and thus these all have the same degree. 
 As the maximum value for $|L\cap[j_r]|$ is $\min\{i_r,j_r\}$, we have 
 \[ 
   \snd(g)\ =\ \sum_{|L\cap[j_r]|=\min\{i_r,j_r\}} C_L \cdot \snd(\hat{C}_L) 
   \ \ +\ \sum_{|L\cap[j_r]|=\min\{i_r,j_r\}-1} C_L \cdot \topd(\hat{C}_L)\ . 
 \] 
 The same arguments as before show that there is no cancellation in these sums. 
 In particular, the second sum is nonempty and nonzero, which implies that $\snd(g)\neq 0$. 
\end{proof} 

\begin{lemma} \label{lemma3} 
 Let $g$ be the determinant of  $M(\ibull; \jbull)$ and assume that $g$ is nonzero, irreducible, and inhomogeneous. 
 Then $\snd(g)\neq 0$ and $\topd(g)$ is coprime to $\snd(g)$. 
\end{lemma} 
\begin{proof} 
 Since $g$ is inhomogeneous, $r> 1$ and $\snd(g)\neq 0$, by Corollary~\ref{corTop}. 
 If $\Upsilon(\ibull; \jbull) =\emptyset$, then $\topd(g)$ is irreducible by Lemma \ref{lemma2} and 
 thus is coprime to $\snd(g)$ as it has greater degree. 
 
 Now suppose that $\Upsilon(\ibull; \jbull)\neq\emptyset$, so that $\topd(g)$ is reducible, and that one of its factors 
 divides $\snd(g)$. 
 We use the notation of Corollary~\ref{corred}. 
 Suppose that for some $t\in\{0,\dotsc,m\}$, we have $\snd(g)=h f_t$, for some polynomial $h$. 
 Here, $f_t=\topd(g_t)$, where $g_t$ is the determinant of the submatrix $\defcolor{M_t}:=M(\ibull; \jbull)_t$ of 
 $M(\ibull; \jbull)$ as defined in~\eqref{Eq:M_t}. 
 
 Suppose that $M_t$ has columns indexed by the interval $[a,b]$ and rows by $[c,d]$, 
 and $n:=b-a+1$ is its size. 
 Let us consider the expansion of $g=\det M(\ibull; \jbull)$ along the rows $[c,d]$ of $M_t$, 
 \begin{equation}\label{Eq:expansion} 
   g\ =\ C_{[a,b]}\hat{C}_{[a,b]}\ +\ 
   \sum_{L\in\binom{[N]}{n}\ ,\ L\neq[a,b]} C_L\hat{C}_L\,. 
 \end{equation} 
 Suppose that $\defcolor{\delta}:=\deg(g)$. 
 By   Corollary~\ref{corred}, only the first term 
 in~\eqref{Eq:expansion} has degree $\delta$. 
 As in Section~\ref{S:MainResult}, let $p_{\delta-1}$ be the homogeneous component of degree $\delta{-}1$ in the polynomial 
 $p$. 
 Then 
 \begin{equation}\label{Eq:SND_expansion} 
   h f_t\ =\ \snd(g)\ =\ 
    \snd(C_{[a,b]}\hat{C}_{[a,b]})\ +\ 
   \sum_{ L\neq[a,b]} \bigl(C_L\hat{C}_L\bigr)_{\delta-1}\,. 
\end{equation} 
 If we specialize the indeterminates not appearing in $C_{[a,b]}\hat{C}_{[a,b]}$ to zero, we obtain 
\[ 
  \overline{h} f_t\ =\ \snd(C_{[a,b]}\hat{C}_{[a,b]})\ =\ 
   \topd(C_{[a,b]})\snd(\hat{C}_{[a,b]})\ +\ \snd(C_{[a,b]})\topd(\hat{C}_{[a,b]})\ , 
\] 
 where \defcolor{$\overline{h}$} is the specialization of $h$. 
 Since $f_t=\topd(C_{[a,b]})$ is irreducible and $\topd(\hat{C}_{[a,b]})\neq 0$, we conclude that $\snd(C_{[a,b]})=0$. 
 Since $C_{[a,b]}=\det M_t$, and it is irreducible (as $\topd(C_{[a,b]})$ is irreducible), 
 Corollary~\ref{corTop} implies that $1{+}s_t=s_{t+1}$ so that $M_t=A^{s_{t+1}}$ is a square matrix of indeterminates. 
 Thus $f_t=\det M_t$ and every term of $f_t$ involves a variable from each column of $M_t$. 
 
 The only variables from rows in $[c,d]$ in terms of $g$ come from $C_{[a,b]}$ and the minors $C_L$ in the sum 
 in~\eqref{Eq:expansion}. 
 Each $C_L$ for $L\neq [a,b]$ is the determinant of a matrix with at least one column not from among $[a,b]$, 
 consequently, there is no term of $C_L$ and hence of $C_L\hat{C}_L$ that involves a variable from each column of $M_t$. 
 This implies that $f_t$ cannot divide the sum of~\eqref{Eq:SND_expansion}, and thus no term in the 
 sum of~\eqref{Eq:expansion} has degree $\delta{-}1$. 
 We will show that the sum of~\eqref{Eq:expansion} has degree $\delta{-}1$, which is a contradiction. 
 This will imply that $\topd(g)$ is coprime to $\snd(g)$ and complete the proof. 
 
 Observe that the matrix $M(\ibull;j\bull)$ has the following block form 
 \[ 
   \left( \begin{matrix} 
                *          &  *  & M(\ibull';\jbull')\\ 
                *          & M_t &          P        \\ 
      M(\ibull'';\jbull'') &  Q  &          0 
       \end{matrix}\right)  \ , 
 \] 
 where $\ibull'=i_1,\dotsc,i_{s_t}$ and $\ibull''=i_{1+s_{t+1}},\dotsc,i_r$, and the same for $\jbull'$ and $\jbull''$. 
 Both $P$ and $Q$ have block structure $\left(\begin{smallmatrix}I&0\\ 0&0\end{smallmatrix}\right)$, where 
 $I=I_{i_{1+s_t},i_{s_t}}$ for $P$ and $I=I_{i_{1+s_{t+1}},i_{s_{t+1}}}$ for $Q$. 
 If $t=0$, then $M(\ibull';\jbull')$ and its rows and columns are omitted, while if $t=m$, then 
 $M(\ibull'';\jbull'')$ and its rows and columns are omitted, but at most one of these occurs, as 
 $m\geq 1$. 
 Note that $\hat{C}_{[a,b]}=\det M(\ibull';\jbull') \cdot \det M(\ibull'';\jbull'')$. 
 
 If $t\neq m$, then $a>1$ and let $L:=\{a{-}1\},(a,b]$. 
 Then $C_L$ is the determinant of the matrix obtained from $M_t$ by replacing its first column of variables with another 
 column of variables, so $\deg C_L=\deg C_{[a,b]}$. 
 Similarly, $\hat{C}_L$ is the product $\det M(\ibull';\jbull')\cdot \det M$, where $M$ is obtained from 
 $M(\ibull'';\jbull'')$ by replacing its last column with the first column of $Q$. 
 This amounts to setting all variables in the last column of $M(\ibull'';\jbull'')$ to zero, except for the first, 
 which is set to 1. 
 This variable was in the block $A^{1+s_{t+1}}$, and by Corollary~\ref{corVars} it appears in 
 $\topd(\det M(\ibull'';\jbull''))$. 
 This implies that the degree of $C_L \hat{C}_L$ is $\delta{-}1$. 
 
 If $t\neq 0$, then $b<N$ and we let $L=[a,b),\{b{+}1\}$. 
 We have $\deg \hat{C}_L=\deg \hat{C}_{[a,b]}$, as they are determinants of matrices of the same format. 
 However, $C_L$ is obtained from $C_{[a,b]}$ by setting all variables in the last column of $M_t$ to zero, except the for 
 the first, which is set to 1. 
 This again implies that the degree of $C_L \hat{C}_L$ is $\delta{-}1$, which 
 shows that the sum of~\eqref{Eq:expansion} has degree $\delta{-}1$, and completes the proof. 
\end{proof} 

\begin{proof}[Proof of Lemma~\ref{lemma:Irr}] 
 Recall that we are considering $\Flndot$. 
 Set $\defcolor{a_1}:=n_1$, $\defcolor{a_t}:=n_t-n_{t-1}$ for $t=2,\ldots r$, and $\defcolor{a_{r+1}}:=n{-}n_r$. 
 Let us augment the coordinates~\eqref{eqnmatrix} to a square matrix by appending 
 $(I_{a_{r+1}}\ \ 0)$ in the remaining rows as follows: 
 \begin{equation}\label{Eq:augmented} 
   \left(\begin{matrix} 
     *&*&\cdots&*&*& I_{a_1}\\ 
     *&*  & \cdots &* & I_{a_2} & 0\vspace{-4pt} \\ 
    \vdots &\vdots &\iddots &\iddots &0 &0\vspace{-6pt}\\ 
     \vdots &* &\iddots &\iddots &0      &0\vspace{-6pt}\\ 
    *&I_{a_{r}}&  0 &\iddots &0 &0 \\ 
     I_{a_{r+1}} &0 &0 &\dots&0 &0 
   \end{matrix}\right)_{n\times n}. 
 \end{equation} 
 For $a\in [n]$, the divisor $D_a$ is given by the $a\times a$ principal minor \defcolor{$f^{(a)}$} of this matrix. 
 Each minor $f^{(a)}$ is nonzero and irreducible as $D_a$ is irreducible. 
 For $a\leq \min\{n_r, n{-}n_1\}$, the $a\times a$ principal minor is the determinant of the first $a$ rows and $a$ columns 
 of~\eqref{Eq:augmented}, and thus has the form~\eqref{eqnmatrix}. 
 If $\min\{n_r, n-n_1\}<a\leq n$, then the matrix formed by the first $a$ rows and $a$ columns of~\eqref{Eq:augmented} 
 does not have this form. 
 When $n{-}n_1<a$, its last $a+n_1-n$ columns have an identity matrix in the first  $a+n_1-n$ rows and 0s elsewhere, 
 and when $n_r<a$, its last $a-n_r$ rows have an identity matrix in the first $a-n_r$ columns and 0s elsewhere. 
 
 In the first case, removing the first $a+n_1-n$ rows and last $a+n_1-n$ columns does not change the determinant, and in 
 the second case, removing the first $a-n_r$ columns and $a-n_r$ columns does not change the determinant. 
 After these removals, we are left with a matrix having the form~\eqref{eqnmatrix}. 
 Hence, by Lemma \ref{lemma1} and Corollary \ref{corred}, every polynomial $\topd(f^{(a)})$ is either irreducible or a 
 product of distinct irreducible polynomials, and hence is square-free. 
 By Lemma \ref{lemma3},  $\topd(f^{(a)})$ and $\snd(f^{(a)})$ are coprime whenever $f^{(a)}$ is not homogeneous (in which 
 case $\snd(f^{(a)})\neq 0$). 
 This proves statement (1) of Lemma~\ref{lemma:Irr}. 
 
 For statement (2), let us first consider the irreducible factorization of $\topd(f^{(a)})$ for $a\in[n{-}1]$. 
 Let \defcolor{$M^{(a)}$} be the principal $a\times a$ submatrix of~\eqref{Eq:augmented}. 
 By Corollary~\ref{corred}  and Remark~\ref{R:reformulation}, the factorization of $\topd(f^{(a)})$ is determined by 
 decompositions of $M^{(a)}$ as in~\eqref{Eq:Block}. 
 That is, by the rows of $M^{(a)}$ whose $2\times 2$ block along the anti-diagonal is 
 $\left(\begin{smallmatrix}*&*\\ *&1\end{smallmatrix}\right)$. 
 From the form of~\eqref{Eq:augmented} this occurs when the northwest 1 of some $I_{a_s}$ is in the indicated position. 
 In this case, $a+a_s=n$ and it occurs in row $n_s+1$ and column $a-n_s+1$. 
 
 Thus each row $n_s$ giving the block structure of~\eqref{Eq:augmented} will contribute to the factorization of a unique 
 $\topd(f^{(a)})$, namely when $a=n-a_s$. 
 Suppose that 
 \[ 
   \topd(f^{(a)})\ =\ f^{(a)}_0\cdot f^{(a)}_1 \dotsb f^{(a)}_{m_a} 
 \] 
 is the irreducible factorization of $\topd(f^{(a)})$. 
 Here, $m_a$ is the number of indices $s$ such that $a+a_s=n$ and $\defcolor{f^{(a)}_i}:=\topd(\det(M^{(a)}_i))$, where 
 \defcolor{$M^{(a)}_i$} is the corresponding submatrix of $M^{(a)}$. 
 These matrices $M^{(a)}_0,\dotsc,M^{(a)}_{m_a}$ lie along the anti-diagonal of $M^{(a)}$ between adjacent rows 
 $n_s,n_{s'}$ such that $a_s=a_{a'}=n-a$ (or row 1 when $i=0$ or row $a$ when $i=m_a$). 
 
 Statement (2) follows from the claim that if $a\neq b$, then for all $i,j$, $f^{(a)}_i\neq f^{(b)}_j$, as these are 
 irreducible. 
 To prove the claim, let $M'$ be the matrix $M^{(a)}_i$, after removing rows and columns coming from $I_{a_1}$ if $a+a_1>n$ 
 and $i=0$ and after removing rows and columns corresponding to $I_{a_{r+1}}$ if $a>n_r$ and $i=m_a$. 
 Then $M'$ has structure as in~\eqref{eqnmatrix} and by Corollary~\ref{corVars} each variable of each anti-diagonal block 
 $A^t$ of $M'$ appears in $f^{(a)}_i$. 
 The claim now follows, as this set of variables is different for distinct $a$ and $i$. 
\end{proof} 

\section*{Acknowledgments} 
The  
 authors would like to  thank  Kwok Wai Chan and Kentaro Hori for helpful discussions. 
C. Li is  supported      by NSFC Grants 11822113, 11831017 and 11771455. 
Sottile's work is supported by grant 636314 from the Simons Foundation. 
 \providecommand{\bysame}{\leavevmode\hbox to3em{\hrulefill}\thinspace}
\providecommand{\MR}{\relax\ifhmode\unskip\space\fi MR }
\providecommand{\MRhref}[2]{%
  \href{http://www.ams.org/mathscinet-getitem?mr=#1}{#2}
}
\providecommand{\href}[2]{#2}

\end{document}